\documentclass[11pt]{amsart}

\usepackage{color}
\usepackage{amsfonts, amsmath, amsfonts, amssymb}
\definecolor{RedClr}{rgb}{1,0,0}
\definecolor{BlueClr}{rgb}{0,0,1}
\definecolor{TextColor}{rgb}{0,0,0.5}
\definecolor{Violet}{rgb}{0.5,0,1}
\definecolor{Bordeaux}{rgb}{1,0.3,0.4}

\begin{document}

\newtheorem{thm}{Theorem}[section]
\newtheorem{cor}[thm]{Corollary}
\newtheorem{lem}[thm]{Lemma}
\newtheorem{prop}[thm]{Proposition}
\theoremstyle{definition}
\newtheorem{defn}[thm]{Definition}
\theoremstyle{remark}
\newtheorem{rem}[thm]{Remark}
\numberwithin{equation}{section} \theoremstyle{quest}
\newtheorem{quest}[]{Question}
\numberwithin{equation}{section} \theoremstyle{prob}
\newtheorem{prob}[]{Problem}
\numberwithin{equation}{section} \theoremstyle{answer}
\newtheorem{answer}[]{Answer}
\numberwithin{equation}{section}
\theoremstyle{fact}
\newtheorem{fact}[thm]{Fact}
\numberwithin{equation}{section}
\theoremstyle{facts}
\newtheorem{facts}[thm]{Facts}
\numberwithin{equation}{section}
\newenvironment{prf}{\noindent{\bf Proof}}{\\ \hspace*{\fill}$\Box$ \par}
\newenvironment{skprf}{\noindent{\\ \bf Sketch of Proof}}{\\ \hspace*{\fill}$\Box$ \par}

\title[]{On the properties of the combinatorial Ricci flow for surfaces}
\author{Emil Saucan}

\address{Department of Mathematics, Technion, Haifa, Israel}%
\email{semil@tx.technion.ac.il}%

\thanks{Research supported by 
by European Research Council under the European Community's Seventh Framework Programme
(FP7/2007-2013) / ERC grant agreement n${\rm ^o}$ [203134].}%

\date{\today}

\begin{abstract}
We investigate the properties of the combinatorial Ricci flow for surfaces, both forward and backward -- existence, uniqueness and singularities formation. We show that the positive results that exist for the smooth Ricci flow also hold for the combinatorial one and that, moreover, the same results hold for a more general, metric notion of curvature. Furthermore, using the metric curvature approach, we show the existence of the Ricci flow for polyhedral manifolds of piecewise constant curvature. We also study the problem of the realizability of the said flow in $\mathbb{R}^3$.
\end{abstract}

\maketitle


\section{Introduction and first results}

We consider the combinatorial Ricci flow introduced by Chow and Luo \cite{CL}, but we approach it not directly from its definition, but rather via the ``classical'' Ricci flow (i.e. the Ricci flow for smooth surfaces). While this approach is certainly less elegant, it allows us to obtain, by making appeal to the power of the classical theory (see \cite{Ha}, \cite{To}, \cite{To1}), a number of results that appear not to have been proved previously: existence of the reverse flow, uniqueness, singularities formation and the issue of embeddability (realizability) in $\mathbb{R}^3$.
Such results are not less important from the viewpoint of the applications of the combinatorial flow -- see, e.g. \cite{Gu-Yau}, \cite{JKG}.

We consider first compact surfaces without boundary and we concentrate mainly 
on polyhedral objects, since these are those arise more naturally and also that are of most interest in applications. We shall also show 
that by considering a more general notion of {\it metric curvature}, we easily obtain
an extension of these results to a larger class of geometric objects (and, in particular, to $CW$ complexes, not necessarily regular).\footnote{The class of $CW$ complexes for which this generalization holds is larger than the one considered in \cite{CL}, since we do not have to restrict ourselves to complexes such that each to cell has 3 vertices and each vertex has degree 2.}

We start with the following basic fact: Given a polyhedral surface $S^2_{Pol}$, we can find a smooth surface $S^2_{s}$ arbitrarily close to it, both in the Hausdorff metric and as far as Gaussian curvature is concerned. More precisely, we have the following result due to Brehm and  K\"{u}hnel \cite{BK}

\begin{prop}[Brehm and  K\"{u}hnel \cite{BK}, Proposition 1] \label{prop:BKapprox}
Let $S^2_{Pol}$ be a compact polyhedral surface without boundary. Then there exists a sequence $\{S^2_m\}_{m \in \mathbb{N}}$ of smooth surfaces, (homeomorphic to $S^2_{Pol}$), such that it assures good approximation of $S^2_{Pol}$ in the
\begin{enumerate}
\item {\em metric sense}, that is
\begin{enumerate}
\item $S^2_m = S^2_{Pol}$ outside the $\frac{1}{m}$-neighbourhood of the 1-skeleton of $S^2_{Pol}$,
\item The sequence $\{S^2_m\}_{m \in \mathbb{N}}$ converges to 
$S^2_{Pol}$ in the Hausdorff metric;
\end{enumerate}
\item {\em curvature sense}, more precisely the (combinatorial) curvatures of $S^2_m$ weakly converge to $S^2_{Pol}$ in the sense of measures. Here the curvature measure in the smooth case is the area measure weighted by Gauss curvature (considered as a function) and in the polyhedral case it is the Dirac measure at the vertices weighted by) the combinatorial curvature.
\end{enumerate}
\end{prop}

Recall that combinatorial Gauss curvature for polyhedral surfaces is defined by the {\it angular defect}, that is:
\begin{equation}
K(p) = 2\pi - \sum_{i=1}^{m_p}\alpha_i(p)\,
\end{equation}
where $\alpha_1,\ldots,\alpha_{m_p}$ are the (interior) face angles adjacent to the vertex $v_i$.

The combinatorial Ricci flow is then given as
\begin{equation} \label{eq:RicciComb}
\frac{dr_i}{dt} = - 2K_ir_i\,,
\end{equation}
in straight forward analogy with the classical flow
\begin{equation} \label{eq:RicciSmooth}
\frac{dg_{ij}(t)}{dt} = - 2K(t)g_{ij}(t)\,.
\end{equation}

(In using proper, rather than partial derivatives notation, we follow \cite{CL}.)

We also consider the close relative of (\ref{eq:RicciSmooth}), and its combinatorial version, namely in the normalized flow
\begin{equation} \label{eq:RicciFlow+}
\frac{dg_{ij}(t)}{dt} = - 2(K - K(t))g_{ij}(t)\,,
\end{equation}
where $K$ is the average sectional (Gauss) curvature of the initial 
surface $S_0$, $K = \int_{S_0}K(t)dA\big/\int_{S_0}dA$.
\marginpar{\tiny \bf }

Here, given the (Euclidean) triangulation with edges $e_{ij}$, of lengths $l_{ij}$, the numbers $r_i$ satisfy the following relation:
 \begin{equation}
l_{ij} = \sqrt{r_i2 + r_j^2 + 2r_ir_j\cos(\Phi(e_{ij})}\,,
\end{equation}
where the $r_i$'s represent the radii of the circle packing defined by the edge lengths above and by the intersection angles $\Phi(ij) = \Phi(e_{ij})$ between the adjacent circles of the packing.

Note that, if starting with a polyhedral surface $S_{Pol}$, 
in order that (\ref{eq:RicciSmooth}), (\ref{eq:RicciFlow+}) should hold (and, indeed, make sense as a flow),
it is important to approximate  $S_{Pol}$ both in the metric sense and also that curvature is well approximated, hence the interest in -- and the need for -- Proposition \ref{prop:BKapprox}.

A proof of Proposition \ref{prop:BKapprox} is given in \cite{BK} (where a stronger result is also proved). In a nutshell, it proceeds as follows: Excise small enough (disjoint) disk neighbourhoods of each of the vertices. The exterior of the union of such neighbourhoods is clearly approximable by a smooth surface composed of flat regions (the interior of the faces) and zero curvature cylinders (around the edges). Each of the neighbourhoods can be approximated by a smooth surface $S(p)$, such that $\partial S(p)$ is a geodesic $m_p$-gon, such that the sum of its exterior angles is $m_p\pi - \sum_i\alpha_i(p)$. Obviously, as noted by Brehm and  K\"{u}hnel, the 
proof above extends to surfaces only {\it locally embedded} in $\mathbb{R}^3$. (See also Remark \ref{rem:last} below regarding the existence of such local embeddings.) Let us also note the fact that -- a fortiori, given the freedom in the construction, provided by the supplementary  dimensions --  it also holds for surfaces embedded in $\mathbb{R}^N$, for some $N > 3$.\footnote{A polyhedral model of the {\it Boy surface} \cite{Br} provides a good illustrating example.} This is important if an abstract surface is given and when one must start by considering first an isometric embedding in  $\mathbb{R}^N$. (See \cite{Sa}, \cite{Sa1} for a discussion on the feasibility and practicability of such embeddings for $PL$, as well a more general types of surfaces, both in the applied and purely mathematical contexts.)

However, we do not content with the previous demonstration, but we also provide a different proof, that also applies 
to a larger class of surfaces. Moreover, it captures 
better the meaning of the notion of Hausdorff convergence and its interplay 
with curvature.

Our approach is less elementary, but in a sense more natural for geometers and topologists, as well as for image processing applications, and that uses only standard analytic tools: Instead of building the smooth surfaces from a set of ``standard elements'' (cylinders, etc.), we consider instead {\it smoothings} $S_m^2$ (see, e.g. \cite{Mun}).\footnote{A similar approach of approximating discrete structures by smooth ones is adopted also in theoretical physics \cite{FL1}, \cite{FL2}, the paradigm therein being that the structure of space-time at the smallest scales is, in fact, discrete and that classical models are smooth approximations of these structures.}
Since, by \cite{Mun}, Theorem 4.8, such smoothings are $\delta$-{\it approximations}, and therefore, for $\delta$ small enough, also $\alpha$-approximations 
of the given piecewise-linear surface $S^2_{Pol}$, they approximate arbitrarily well both distances and angles on $S^2_{Pol}$. (Not to encumber the presentation with too many details regarding tools of differential topology, we have concentrated the relevant definitions, and results in an appendix.) Therefore
angles, hence defects, are arbitrarily well approximated as well.
While Munkres' results concern $PL$ manifolds, they can be applied to polyhedral ones as well, because, by 
definition, polyhedral manifolds have simplicial subdivisions (and furthermore, such that all vertex links are combinatorial manifolds) -- see, e.g. \cite{BS}, p. 346.
The fact that 2-dimensional manifolds with a $CW$ complex structure are also smoothable follows from the fact that any manifold of dimension $\leq 3$ admits a $PL$ structure (see, e.g. \cite{Th}) and that, furthermore, this structure admits a unique smoothing (see, e.g. \cite{Mun}).
In consequence, Gauss curvature of the smooth surface approximates arbitrarily well combinatorial curvature, at the essential common points (i.e. the vertices of the given polyhedral surface).
We note that, since, by \cite{Ha}, Corollary 5.2 (see also \cite{CK}, Proposition 5.4) the Ricci flow is conformal, it follows that combinatorial curvature approximates arbitrarily well the curvature of the evolved surfaces, at any time $t$ (see Section 2 below for details).

Note that our proof renders in fact a somewhat stronger result than that of \cite{BK}, since no embedding in $\mathbb{R}^3$
is apriorily assumed, just in some $\mathbb{R}^N$; however, as we have already seen, this represents only a slight improvement.
More importantly, no change in the geometry of the 1-skeleton is made, not even in the neighbourhoods of the vertices.

Moreover, it follows that {\it metric quadruples} (see definition below) on $S_{Pol}$ are also arbitrarily well approximated (including their angles) by the corresponding metric quadruples) on $S_m$. But, by \cite{Wa} (see also  \cite{BM}, Theorems 11.2 and 11.3), the {\it Wald metric curvature} (see below) of $S_m$, at a point $p$, $K_W(p)$ equals the classical (Gauss) curvature $K(p)$. Hence the Gauss curvature of the smooth surfaces $S_m$ approximates arbitrarily well the metric one of $S_{PL}$ (and, as in \cite{BK}, the smooth surfaces differ from polyhedral one only on (say) the $\frac{1}{m}$-neighbourhood of the 1-skeleton of $S_{Pol}$.
This statement can be made even more precise, by assuring that the convergence is in the Hausdorff metric. This follows from results of Gromov \cite{Gr-carte} -- see \cite{SA} for details.  
That such curvatures converge not only punctually, but as measures as well, i.e. that the so called ${\rm CCP}(K)$ property of \cite{BK}, Proposition 1,\footnote{this corresponds, essentially,  to Condition (2) in Theorem \ref{prop:BKapprox} above} also holds, follows, as a particular case, from \cite{CMS}, Theorem 5.1, using the fact that polyhedral manifolds represent secant approximations of their own smoothings.

Here {\it metric quadruples} are defined as follows:

\begin{defn}
Let $(M,d)$ be a metric space, and let $Q = \{p_1,...,p_4\} \subset M$, together with the mutual distances:
$d_{ij} = d_{ji} = d(p_i,p_j); \, 1 \leq i,j \leq 4$. The set $Q$ together with the set of distances
$\{d_{ij}\}_{1\leq i,j \leq 4}$ is called a {\it metric quadruple}.
\end{defn}

\begin{rem}
Metric quadruples can be defined in a slightly more abstract manner, without the aid of the ambient
space: a metric quadruple being defined, in this approach, as a $4$ point metric space; i.e. $Q = \big(\{p_1,...,p_4\}, \{d_{ij}\}\big)$, where
the distances $d_{ij}$ verify the axioms for a metric.
\end{rem}

Before we can define the notion of embedding curvature, we have first to introduce some notation: Let $S_{\kappa}$ denote the
complete, simply connected surface of constant Gauss curvature $\kappa$, i.e. $S_{\kappa} \equiv \mathbb{R}^2$,
if $\kappa = 0$; $S_{\kappa} \equiv \mathbb{S}^2_{\sqrt{\kappa}}$\,, if $\kappa
> 0$; and $S_{\kappa} \equiv \mathbb{H}^2_{\sqrt{-\kappa}}$\,, if $\kappa < 0$. Here $S_{\kappa} \equiv
\mathbb{S}^2_{\sqrt{\kappa}}$ denotes the sphere of radius  $R = 1/\sqrt{\kappa}$, and $S_{\kappa} \equiv
\mathbb{H}^2_{\sqrt{-\kappa}}$ stands for the hyperbolic plane of curvature $\sqrt{-\kappa}$, as represented by
the Poincar\'{e} model of the plane disk of radius $R = 1/\sqrt{-\kappa}$\,.

\begin{defn}
 The {\em embedding curvature} $\kappa(Q)$ of the metric quadruple $Q$ is defined to be the curvature $\kappa$ of the gauge surface
$S_{\kappa}$ into which $Q$ can be isometrically embedded.
\end{defn}

We are now able to bring the definition of {\it Wald curvature} \cite{Wa} (or rather of its modification due to Berestovskii \cite{Ber}):

\begin{defn} \label{def:WBcurv}
Let $(X,d)$ be a metric space. 
An open set $U \subset X$ is called a {\it region of curvature} $\geq \kappa$ iff any metric quadruple can be isometrically embedded in $S_m$, for some $m \geq k$.
A metric space $(X,d)$ is said to have 
{\it Wald-Berestovskii curvature} $\geq \kappa$ iff for any $x \in X$ is contained in a region $U$ of curvature $\geq \kappa$.
\end{defn}

\begin{rem}
While the second part of the definition above is not needed in the remainder of the paper,
we bring it for completeness (and for its importance elsewere -- see, e.g. \cite{Sa}, \cite{Sa1}).
\end{rem}

\begin{rem}
Note that we can consider the Wald-Berestovskii curvature at an accumulation point of a metric space, hence on a smooth surface, by considering limits of the curvatures of (nondegenerate) regions of diameter converging to 0.
\end{rem}

\begin{rem}
The Wald-Berestovskii of a metric can actually be computed, using the following formula for the embedding curvature of a metric quadruple
\begin{equation} \label{eq:k(Q)}
 \kappa(Q) = \left\{
         \begin{array}{clclcrcr}
           \mbox{0} &  \mbox{if $D(Q) = 0$\,;} \\
           \mbox{$\kappa,\, \kappa < 0$} & \mbox{if $det({\cosh{\sqrt{-\kappa}\cdot d_{ij}}}) = 0$\,;} \\
           \mbox{$\kappa,\, \kappa > 0$} & \mbox{if $det(\cos{\sqrt{\kappa}\cdot d_{ij}})$ and $\sqrt{\kappa}\cdot d_{ij} \leq
           \pi$}\\
           & \mbox{\,\, and all the principal minors of order $3$ are $\geq 0$;}
         \end{array}
   \right.
\end{equation}
where $d_{ij} = d(x_i,x_j), 1 \leq i,j \leq 4$, and $D(Q)$ denotes the so called {\it Cayley-Menger determinant}:
\begin{equation}                                \label{eq:D}
 D(x_1,x_2,x_3,x_4) = \left| \begin{array}{ccccc}
                                            0 & 1 & 1 & 1 & 1 \\
                                            1 & 0 & d_{12}^{2} & d_{13}^{2} & d_{14}^{2} \\
                                            1 & d_{12}^{2} & 0 & d_{23}^{2} & d_{24}^{2} \\
                                            1 & d_{13}^{2} & d_{23}^{2} & 0 & d_{34}^{2} \\
                                            1 & d_{14}^{2} & d_{24}^{2} & d_{34}^{2} & 0
                                      \end{array}
                               \right|\;.
\end{equation}
However, it should be noted that, as far as the 
actual computation of $\kappa(Q)$ using Formula (\ref{eq:k(Q)}) is concerned, the equations involved are -- apart from the Euclidean case --  transcendental, are not solvable, in general, using elementary methods. Moreover, when solving them by with the assistance of computer assisted  methods, 
they display certain numerical instability. For a more detailed discussion and some first numerical results, see \cite{Sa04}, \cite{SA}.
\end{rem}

Evidently, in the context of polyhedral surfaces, the natural choice for the open set $U$ required in Definition \ref{def:WBcurv} is the closed {\it star} of a given vertex $v$, that is, the set $\{e_{vj}\}_j$ of edges incident to 
$v$. Therefore, for such surfaces, the set of metric quadruples 
containing the vertex $v$ is finite.

The combinatorial (defect) and metric (Wald-Berestovskii) notions of curvature are more closely related than just by having as limit, when the mesh of the polyhedral approximation of a smooth surface tends to zero, the Gauss curvature of the said (smooth) surface. Indeed, Wald-Berestovskii curvature can be characterized in terms of angles' sum at a vertex (hence defect). Before formally stating this result, we first need to introduce some further notation:

Given three points $x_i,x_j,x_l$ in a metric space $(X,d)$, we denote by $\alpha_\kappa(x_i,x_j,x_l)$ $\in [0,\pi]$, the angle $\measuredangle(x_jx_ix_l)$\footnote{that is of apex $x_i$} of the model triangle in $S^n_\kappa$.
%
%
Let $Q = \{x_1,x_2,x_3,x_4\}$ be a {\it metric quadruple}. 
We introduce the following quantity associated with $Q$: 
\begin{equation} \label{eq:Vkappa}
V_\kappa(x_i) = \alpha_\kappa(x_i;x_j,x_l) + \alpha_\kappa(x_i;x_j,x_m) + \alpha_\kappa(x_i;x_l,x_m)\,
\end{equation}
where $x_i,x_j,x_l,x_m \in Q$ are distinct, and $\kappa$ is any number.
%
%


We can now bring the promised characterization of Wald-Berestovskii in terms of angle sum:

\begin{prop}[\cite{Pl}, Theorem 23] \label{prop:wald-vs-Vkappa} 
Let $(X,d)$ be a metric space and let $U \in X$ be an open set. $U$ is a region of curvature $\geq \kappa$ iff $V_\kappa(x) \leq 2\pi$, for any metric quadruple $\{x,y,z,t\} \subset U$.
\end{prop}

The result 
above shows that, in fact, the metric approach to curvature is equivalent to the combinatorial (angle-based) one, as far as polyhedral surfaces (in $\mathbb{R}^3$) are concerned. (In fact, the metric approach is more general, since it can be applied
to a very large class of 
metric spaces. Also, we should again emphasize that, as far as approximations of smooth surfaces in $\mathbb{R}^3$ are concerned, both approaches render,
in the limit, the classical Gauss curvature.)


Note that 
we just gave
a positive answer to the question --  unposed so far, to the best of our knowledge -- whether the metric curvature version of Brehm and H\"{u}hnels's basic result also holds, namely we have proved:


\begin{prop} \label{prop:Waldapprox}
Let $S^2_{Pol}$ be a compact polyhedral surface without boundary. Then there exists a sequence $\{S^2_m\}_{m \in \mathbb{N}}$ of smooth surfaces, (homeomorphic to $S^2_{Pol}$), such that
\begin{enumerate}
\item
\begin{enumerate}
\item $S^2_m = S^2_{Pol}$ outside the $\frac{1}{m}$-neighbourhood of the 1-skeleton of $S^2_{Pol}$,
\item The sequence $\{S^2_m\}_{m \in \mathbb{N}}$ converges to 
$S^2_{Pol}$ in the Hausdorff metric;
\end{enumerate}
\item $K(S^2_{Pol}) \rightarrow K_W(S^2_{Pol})$, where the convergence is in the weak sense.
\end{enumerate}
\end{prop}

\begin{rem}
The converse implication -- namely that Gaussian curvature $K(\Sigma)$ of a smooth surface $\Sigma$ may be approximated arbitrarily well by the Wald curvatures $K_W(\Sigma_{Pol,m})$ of a sequence of  approximating polyhedral surfaces $\Sigma_{Pol,m}$ -- is, as we have already mentioned above, quite classical.  
(For other approaches to curvatures convergence, see, amongst the extensive literature dedicated to the subject,  \cite{CMS} and \cite{BCM}, \cite{C-SM}, for the theoretical and applicative viewpoints, respectively.)
\end{rem}

We should also stress again the properties of the Gromov-Hausdorf convergence of finite $\varepsilon$-{\it nets}\footnote{Just for the record, recall that, given a metric space $(X,d)$, a $A \subset X$ is called an $\varepsilon${\it-net} iff $d(x,A) \leq \varepsilon$, for any $x \in X$.} in any sequence of approximating surfaces $S^2_m$ (polyhedral or smooth) of a given surface $S^2$ (again, smooth or not).
In particular, by considering $\varepsilon$-nets on surfaces, one automatically ensures (see \cite{Gr-carte}, \cite{BBI}) any {\it intrinsic} geometric property of an approximating to the respective geometric property of the limiting geometric object. Most important for us, this holds for the intrinsic metric and thence for the metric curvatures -- see \cite{Sa04}, \cite{SA} for a more detailed discussion and some numerical experiments.

Also, one can consider simultaneously the combinatorial/metric Ricci flow on a polyhedral surface $S^2_{Pol}$, as well its classical counterpart on its smoothing $S^2$. Since, as we have already noted above, the metric and the classical curvatures, $K_W(S^2_{Pol})$ and $K(S^2_m)$, respectively, are arbitrarily close to each other, and since the equations the two respective flows (combinatorial/metric and classical/smooth) contain the same curvature term, the ensuing metrics at each time during the flow will coincide on the common set, i.e the 1-skeleton of the polyhedral manifold and in the exterior of an arbitrarily small neighbourhood of it. Therefore, the limit surfaces for both flows -- $S^2_{0,Pol}$ and $S^2_0$, respectively -- will be isometric on the said set but perhaps only arbitrarily close to being isometric in the considered neighbourhood.\footnote{One can ensure actual isometry by imposing a certain additional constraint on the so called ``{\it volume density}'' of the surface -- for details, see \cite{Te}.}
Moreover, by considering dense enough $\varepsilon$-nets  (of arbitrarily small mesh), the intrinsic metric of a polyhedral approximation $g_{Pol}$ of a smooth manifold, and the (smooth) metric $g$ of the later will be arbitrarily close to each other. (If one does not consider a limiting process, then the distortion of $g$ by $g_{Pol}$ can be computed, at each time ``$t$''
during the flow, using such formulae as (\ref{eq:metric-dist}) and (\ref{eq:metric-dist2}) below, in conjunction with the computations in, e.g.,  \cite{Pe}, Lemma 3.19.\footnote{For manifolds with boundary, similar distortion estimates follow from the results in \cite{Sa05}. Of course, in this case, one still obtains isometry when restricting to the 1-skeleton.}
This answers to a question posed to us by D. X. Gu.

Before concluding this section, 
we should note that one can control the deformation both of the metric and of (combinatorial) curvature during the Ricci flow -- for details see the proof of Theorem \ref{thm:flow-rate-conv} below.

\section{Main results}
From the ``good'', i.e. metric and curvature, approximations results above, it follows that one can study the properties of the combinatorial Ricci flow via those of its smooth counterpart, by passing to a smoothing of the polyhedral surface. The heavier machinery of metric curvature considered above pay off, in the sense that, by using it, the ``duality'' between the combinatorics of the packings (and angles) and the metric disappears: The flow is purely metric and, moreover, the curvature at each stage (that is, for every ``$t$'') is given -- as in the classical context -- in an intrinsic manner, i.e. solely 
in terms of the metric.

A number of important properties now follow immediately.

\subsection{Existence and uniqueness (forwards and backwards)} In particular, the (local) existence and uniqueness of both the {\it forward} and {\it backward} (i.e. given by $dg_{ij}(t)/dt = 2K(t)g_{ij}(t))$ Ricci flows hold, on some {\it maximal} time interval $[0,T]; 0< T \leq \infty$ (see, e.g. \cite{To}, Theorems 5.2.1 and 5.2.2 and the discussion following them). Beyond the theoretical importance, the existence and uniqueness of the backward flow allow us to find surfaces in the (conformal) class of a given circle packing (Euclidean or Hyperbolic).
More importantly, the use of purely metric approach (based on the Wald curvature or any of other equivalent metric curvatures), rather than the combinatorial (and metric) approach of \cite{CL}, allows us to give a first, theoretical at this point, answer to Question 2, p. 123, of \cite{CL}, namely whether there exists a Ricci flow defined on the space of all piecewise constant curvature metrics (obtained via the assignment of lengths to a given triangulation of 2-manifold). Since, by Hamilton's results \cite{Ha} (and those of Chow \cite{Ch}, for the case of the sphere), the Ricci flow exists for all compact surfaces, it follows from our arguments above that the fitting metric (combinatorial) flow exits for surfaces of piecewise constant curvature. In consequence, given a surface of piecewise constant curvature (e.g. a mesh with edge lengths satisfying the triangle inequality for each triangle), one can evolve it by the Ricci flow, either forward, as in works discussed above, to obtain, after the suitable area normalization, the polyhedral surface of constant curvature conformally equivalent to it; or backwards, to find the ``primitive'' family of surfaces (including the ``original'' surface) conformally equivalent to the given one. (Here, by ``original'', we mean the surface obtained via the backwards Ricci flow, at time $T$.) It is not necessarily true, however, that all the surfaces obtained via the backwards flow are embedded (or, indeed, embeddable) in $\mathbb{R}^3$ -- for details see Section 3 below.
We can summarize the discussion above as

\begin{prop}
Let $(S^2_{Pol},g_{Pol})$ be a compact polyhedral 2-manifold without boundary, having bounded combinatorial (metric) curvature.
Then there exists $T > 0$ and a smooth family of polyhedral metrics $g(t), t \in [0,T]$, such that
\begin{equation}
\left\{
\begin{array}{ll}
\frac{\partial g}{\partial t} = -2K(t)g(t) & t \in [0,T]\,;\\
g(0) = g_{Pol}\,.
\end{array}
\right.
\end{equation}
(Here $K(t)$ denotes the combinatorial (resp. Wald) curvature induced by the metric g(t).)

Moreover, both the forwards and the backwards Ricci flows have the uniqueness of solutions property, that is, if $g_1(t), g_2(t)$ are two Ricci flows on $(S^2_{Pol}$, such that there exists $t_0 \in [0,T]$ such that $g_1(t_0) = g_2(t_0)$, then $g_1(t) = g_2(t)$, for all $t_0 \in [0,T]$.
\end{prop}

In fact, the existence and uniqueness  of the Ricci flow hold even if we do not restrict to compact surfaces, but we still require that the manifold is complete. Indeed, by applying the ideas employed in the proof above to a result of Shi \cite{Sh} (see also \cite{IMS1}), we obtain the following

\begin{prop}
Let $(S^2_{Pol},g_{Pol})$ be a complete polyhedral surface of metric (combinatorial) curvature bounded from above. Then there exists a (small) $T$ as above, such that there exists a unique solution of (\ref{eq:RicciSmooth}) for any $t \in [0,T]$.
\end{prop}

\begin{rem}
Shi's result (hence the proposition above) does not necessarily hold for noncomplete surfaces -- see \cite{To2}.
\end{rem}

Before proceeding further, we should stress that, at this point, the metric approach introduced above is purely theoretical and, while it allows for a number of (theoretical) results to be inferred from the classical theory, it lacks (at least for now) the simple algorithmic capability of the combinatorial one of Chow and Luo, and certainly of its subsequently development -- see \cite{Gu-Yau}).

\subsection{Convergence rate}
A further type of result, highly important both from the theoretical viewpoint and for computer-driven applications, is that of the convergence rate. For the combinatorial flow, it is shown in \cite{CL} that, in the case of background Euclidean (Theorem 1.1) or Hyperbolic (Theorem 1.2) metric, the solution -- if it exists -- converges, without singularities, exponentially fast to a metric of constant curvature. Using the classical results of \cite{Ha} and \cite{Ch}, we can do slightly better, since we already know that the solution exists and it is unique (see the subsection below for the nonformation of singularities). Moreover, we can control the convergence rate of the curvature:

\begin{thm} \label{thm:flow-rate-conv}
Let $(S^2_{Pol},g_{Pol})$ be a compact polyhedral 2-manifold without boundary. Then the normalized combinatorial (metric) Ricci flow converges to a surface of constant combinatorial (resp. metric) curvature. Moreover, the convergence (rate) is
\begin{enumerate}
\item exponential, if  $K < 0$; $\chi(S^2_{Pol}) < 0$);
\item uniform; if $K = 0$; 
\item exponential, if $K > 0$. 
\end{enumerate}
\end{thm}

\begin{proof}
As already noted, a unique solution for the Ricci flow exists for all $0 < t \leq T$ and, again, these solutions are uniformly (conformally) equivalent. Indeed, by \cite{Ha}, Corollary 5.2 and \cite{CK}, Proposition 5.15, there exists $C= C(\widetilde{g}_{Pol})$, where $\widetilde{g}_{Pol}$ is the smoothing of $g_{Pol}$, such that
\begin{equation} \label{eq:metric-dist}
\frac{1}{C}\widetilde{g}_{Pol} \leq \widetilde{g}_t \leq C\widetilde{g}_{Pol}\,,
\end{equation}
where $\widetilde{g}_t$ is the smoothing of $g_t$, and the discussion above shows that the same holds for the polyhedral metrics.

To estimate the convergence rate for the curvature, we make appeal to the following formulae (see, e.g. \cite{CK}, Proposition 5.18): There exists $C' = C'(g_{Pol}) > 0$, (in fact, $C' = C'(\widetilde{g}_{Pol})$), such that
\begin{enumerate}
\item If $K < 0$ 
then
\begin{equation}\label{eq:curv}
K - C'e^{Kt} \leq K(t) \leq K + C'e^{Kt}\,;
\end{equation}
\item If $K = 0$ 
then
\begin{equation}\label{eq:curv1}
- \frac{C'}{1 + C't} \leq K(t) \leq C'\,;
\end{equation}
\item If $K > 0$ 
then
\begin{equation}\label{eq:curv2}
- C'e^{Kt} \leq K(t) \leq K + C'e^{Kt}\,.
\end{equation}
\end{enumerate}

To show that the metric also converges with exponential rate, one has to 
make appeal to
a refinement of (\ref{eq:metric-dist}), namely that the constant $C$ therein is, in fact, given by $C = e^{2K_{Max}}$, where $K_{Max} = \max|K(t)|\,, t \in [0,T]$\,. (This holds, in fact, for the case of the general Ricci flow, with Gaussian curvature $K$ being replaced, of course, by the Ricci curvature ${\rm Ric}$ -- see, e.g. \cite{To}, Lemma 5.3.2.) In fact, given that the manifold under investigation is compact, hence of curvature bounded below and above, a stronger form of this improvement of (\ref{eq:metric-dist}) can be given, and, moreover, one that is better fitted for the case of polyhedral manifolds (see, e.g. \cite{KL} Lemma 27.1, Remark 27.5 and the following material):
\begin{equation} \label{eq:metric-dist2}
e^{-Kt} \leq \frac{{\rm dist}_{t}(x,y)}{{\rm dist}_{0}(x,y)} \leq e^{Kt}\,,
\end{equation}
where $K_{Max}$ is as above.

By the approximation results above, namely Propositions \ref{prop:BKapprox} and \ref{prop:Waldapprox}, the result follows for the combinatorial and Wald-Berestovskii curvatures, respectively. Alternatively, one can more directly infer the respective convergence rates from (\ref{eq:metric-dist2}) and (\ref{eq:metric-dist}) as far as the metric is concerned, and for the Wald-Berestovskii curvatures from (\ref{eq:metric-dist2}) and (\ref{eq:curv}) -- (\ref{eq:curv2}). Similarly, the convergence for the combinatorial curvature follows again from (\ref{eq:metric-dist2}) in conjunction with (\ref{eq:Vkappa})  and Proposition \ref{prop:wald-vs-Vkappa}.
\end{proof}

\begin{rem}
Existence, uniqueness and, furthermore, convergence rate results for polyhedral can be obtained, using the same techniques as before, for surfaces with (a finite number of) {\it cusps} and {\it funnels}, using quite recent results of Isenberg, Mazzeo and Sesum \cite{IMS1} (for finite area surfaces) and Albin, Aldana and Rochon \cite{AAR} (for surfaces of infinite area). We do not bring here the technical details since they would bring us too far afield -- for details and  further related results, see \cite{IMS2}.
\end{rem}


\subsection{Singularities formation} Another important aspect of any Ricci flow, be it smooth or discrete, is that of singularities formation.
By \cite{CL}, Theorem 5.1, for compact surfaces of genus $\geq 2$, the combinatorial Ricci flow evolves without singularities. However, for surfaces of low genus no such result exists. Indeed, in the case of the Euclidean background metric, that is the one of greatest interest in graphics, singularities do appear \cite{Gu}. Such singularities are always combinatorial in nature and amount to the fact that, at some $t$, the edges of at least one triangle do not satisfy the triangle inequality  \cite{Gu}. These singularities are removed in heuristic manner, using the graphics equivalent of $\varepsilon$-moves (see, e.g. \cite{Ed}). However, by \cite{Ha}, Theorem 1.1, the smooth Ricci flow exists at all times, i.e. no singularities form. By the considerations above, it follows that the metric Ricci flow also exists at all times without the formation of singularities. In fact, by a quite recent result of Topping \cite{To1}, the same result holds even for unbounded (but complete) Riemannian 2-manifolds $(M,g)$ with bounded curvature and satisfying a
certain mild noncollapsing condition, namely that, there exists $r_0 > 0$, such that, for all $x \in M$, the following holds:
\begin{equation} \label{eq:noncolapse}
{\rm Vol}_g\left(B_g(x,r_0)\right) \geq \varepsilon > 0.
\end{equation}
(Here, as usual, $B_g(x,r_0)$ denotes the open ball, in the metric $g$, of center $x$ and radius $r_0$.)

Again, we can 
recap 
the discussion above as

\begin{prop}
Let $(S^2_{Pol},g_{Pol})$ be a complete polyhedral 2-manifold, with at most a finite number of hyperbolic cusps (punctures), having bounded combinatorial (metric) curvature and satisfying the noncollapsing condition (\ref{eq:noncolapse}). Then there exists a unique Ricci flow that contracts
the cusps. Furthermore, the curvature remains bounded at all times during the flow.
\end{prop}

\begin{rem}
Both the boundedness and the noncollapsing conditions evidently hold for surfaces that appear in graphics, hence the fitting result for the combinatorial (and metric) flow also hold for this type of application. It follows, that we can apply also in this context the cusp contracting result of \cite{To1}.
One may, however,  argue that manifolds with (hyperbolic) cusps do not appear in graphics, only compact manifolds (with or without boundary), but in fact many algorithms are modeled upon surfaces with punctures -- see e.g. \cite{Gu-Yau}, \cite{JKG}.
\end{rem}


\section{Embeddability in $\mathbb{R}^3$}

In this section we mainly consider a problem regarding smooth surfaces, and we hope that by now, the connection with its version for polyhedral surfaces is clear. It should be noted, in this context that, by \cite{Mun}, Theorem 8.8, any $\delta$-approximation of an embedding is also an embedding, for small enough $\delta$. Since, as we have already mentioned, smoothing represent $\delta$-approximations, the possibility of using 
results regarding smooth to infer results regarding polyhedral embeddings is proven. (The other direction -- namely from smooth to $PL$ and polyhedral manifolds -- follows from the fact that the {\it secant approximation} (see Appendix) is a $\delta$-approximation if the simplices of the $PL$ approximation satisfy a certain nondegeneracy condition -- see \cite{Mun}, Lemma 9.3.)
We wish to stress here the importance of the embeddability in graphics and image processing. In the only fully implemented Ricci flow, that is the combinatorial flow \cite{Gu-Yau}, \cite{JKG}, the goal is, in fact, to produce, via the circle packing metric, a conformal mapping from the given surface to a surface of constant (Gauss) curvature. Since in the relevant cases  (see \cite{CL}) the surface in question is a planar region (usually a subset of the unit disk), its embeddability (not necessarily isometric) is trivial. Moreover, in the above mentioned works, there is no interest (and indeed, no need) to consider  the (isometric) embeddability of the surfaces $S_t^2$ (see below) for an intermediate time $t \neq 0, T$. However, this aspect is very important if one considers the problem of the Ricci flow for surfaces of piecewise constant curvature; as well as in image processing -- see \cite{ASZ}.

Let $S_0^2$ be a smooth surfaces of positive Gauss curvature, and let $S_t^2$ denote the surface obtained at time $t$ from $S_0^2$ via the Ricci flow. 
For all omitted background material (proofs, further results, etc.) we refer to \cite{HH}.

\begin{prop}
Let $S_0^2$ be the unit sphere $\mathbb{S}^2$, equipped with a smooth metric $g$, such that $K(g) > 0$. Then the surfaces $S^2_t$ are (uniquely, up to a congruence) isometrically embeddable in $\mathbb{R}^3$, for any $t \geq 0$.
\end{prop}

\begin{proof}
By a by now classical result of Nirenberg \cite{Ni} and Pogorelov \cite{Po} (independently), $S_0^2$ admits a smooth isometric embedding in $\mathbb{R}^3$.

The metric $g_t$ being conformal to $g_0$, for any $t$ (see \cite{Ha}), we can write it as 
$g_t = e^{2\varphi}g_0$, for some smooth function $\varphi$.\footnote{This represents, in fact, the first,  basic step in the proof of the Niremberg-Pogorelov theorem.} Then, by elementary computations (see, e.g. \cite{HH}), we obtain that the Gauss curvature of this metric is
\[
K_t = K(g_t) = e^{-2\varphi}(K_{g_0} - t\Delta_{g_0}\varphi) = 
te^{2(1-t)\varphi}K_g + (1 - t)e^{-2t\varphi} > 0\,.
\]
Therefore, again by Nirenberg and Pogorelov's  theorem, the surfaces $S^2_t$ are isometrically embeddable in $\mathbb{R}^3$, for any $t \geq 0$.
\end{proof}

We can, in fact, do somewhat better:

\begin{cor}
Let $S^2_0$ be a smooth surface. If $\chi(S_0^2) > 0$, then there exists some $t_0 \geq 0$, such that the surfaces $S^2_t$ are isometrically embeddable in $\mathbb{R}^3$, for any $t \geq t_0$.
\end{cor}

\begin{proof}
By the continuity of the Gauss curvature during the Ricci flow, it follows that, as some time $t_0$, the $|K_{t_0}| \leq K_0 > 0$. Applying again 
the arguments in 
the proof above, the corollary follows.
\end{proof}

\begin{rem}
In this context it is impossible not to mention Alexandrov's results \cite{Al} regarding convex surfaces in $\mathbb{R}^3$: (A) Any convex surface, endowed with its intrinsic metric, is a manifold of nonnegative curvature; and, in essentially the opposite direction, (B) Any complete metric of positive curvature (nonnegative) on the sphere (the plane)  represents the metric of a (not necessarily smooth) convex surface.

However, we should underline that requiring that a certain polyhedral sphere actually has (strictly) positive curvature at all its vertices it is quite a strong condition; indeed, it is a well known fact in graphics (see, e.g. \cite{LSE}) that even the most standard polygonal approximations of the sphere, exhibit, even at and high resolution, saddle points at certain vertices.
\end{rem}

In contrast with this positive result, the for (complete) surfaces uniformized by the Hyperbolic plane we have the following negative result:

\begin{prop}  \label{prop:non-embedd-hyp}
Let $(S^2_0,g_0)$ be a smooth surface. If $\chi(S^2) < 0$, then there exists some $t_0 \geq 0$, such that the surfaces $S^2_t$ are not isometrically embeddable in $\mathbb{R}^3$, for any $t \geq t_0$.
\end{prop}

\begin{proof}
Since at time $T$, the surface undergoing the flow has constant negative Gauss curvature, it is not smoothly ($\mathcal{C}^4$) embeddable in $\mathbb{R}^3$, by a classical theorem of Hilbert \cite{Hi}. By the continuity of the Gauss curvature during the Ricci flow, it follows that, as some time $t_0$, $K_{t_0} \leq K_0 < 0$. Therefore, by a result of Efimov \cite{Ef}, it follows that $S_{t_0}$ admits no smooth ($\mathcal{C}^2$) isometric immersion (hence embedding) in $\mathbb{R}^3$.
\end{proof}

\begin{rem}
Efimov \cite{Ef1} also proved that even if only the gradient of the Gaussian curvature is bounded (by some specific constant -- see \cite{Ef1}, \cite{HH}), there exists no smooth ($\mathcal{C}^3$) isometric immersion in $\mathbb{R}^3$, and no $\mathcal{C}^2$ isometric immersion exists if the suprema of $K$ and of its gradient are $< \infty$.
On the other hand, Hong \cite{Ho} showed that smooth isometric immersions in $\mathbb{R}^3$ exist if the decay rate at infinity of $K$ is slower that the inverse square of geodesic distance.

Obviously, the above mentioned results render appropriate versions of Proposition \ref{prop:non-embedd-hyp}. However, their applicabilty is of far lesser interest in the context of the Ricci flow for polyhedral surfaces, so we do not formulate them explicitly.
\end{rem}

\begin{rem} \label{rem:last}
In fact, the the situation is far worse, so to speak, than even Hilbert's and Efimov's theorems might suggest. Indeed, no general existence/nonexistence result is available, 
even if one restricts himself to local isometric embedding. Without going in too many details (since this would bring us out of our scope), it is known that such embeddings are possible if $K$ does not vanish; or when $K(p) = 0$ and $dK(p) \neq  0$ or if $dK(p) \geq  0$ in some neighbourhood of a point $p$; and again when $K(p) = 0, dK(p) = 0$ and ${\rm Hess}K(p) < 0$. (For further details and bibliographical references see e.g. \cite{HH}.)

However, as in the global case, (for lower differentiability classes), no general results are even possible. This was first shown by
 Pogorelov \cite{Pog1}, who constructed a $\mathcal{C}^{2,1}$ metric on the unit disk $\mathbb{B}^2 = \mathbb{B}^2(0,1) \subset \mathbb{R}^2$, such that there exists no $\mathcal{C}^{2}$ isometric imbedding in $\mathbb{R}^3$ of $\mathbb{B}^2(0,r)$, for any $0 < r < 1$.
%

We mention in this context that a criterion for the local isometric embedding of polyhedral surfaces in $\mathbb{R}^3$, akin to the classical Gauss fundamental (compatibility) equation in the classical differential geometry of surfaces, was given in \cite{Sa}. Namely, given a vertex $v$, with metric curvature $K_W(v)$, the following system of inequalities should hold:
\begin{equation} \label{eq:metricGauss}
 \left\{
         \begin{array}{lll}
         \max{A_0(v)} \leq 2\pi;\\
         \alpha_0(v;v_j,v_l) \leq \alpha_0(v;v_j,v_p) + \alpha_0(v;v_l,v_p), & {\rm for\; all\;} v_j,v_l,v_p \sim v;\\ 
         V_\kappa(v) \leq 2\pi;
         \end{array}
 \right.
\end{equation}
Here
\begin{equation}
A_0 = \max_i{V_0}\,;
\end{equation}
 ``$\sim$'' denotes incidence, i.e. the existence of a 
connecting edge $e_{i} = vv_j$ and, of course, $V_\kappa(v) = \alpha_\kappa(v;v_j,v_l) + \alpha_\kappa(v;v_j,v_p) + \alpha_\kappa(v;v_l,v_p)$, where $v_j,v_l,v_p \sim v$, etc.

Note that the first two inequalities represent the (extrinsic) embedding condition, while the third one represents the intrinsic curvature (of the $PL$ manifold) at the vertex $v$.

For details and a fitting global embedding criterion see \cite{Sa}.

We mentioned these facts because, beyond the ``archival'' interest (so to speak) in such results, they are relevant if locally embedding based applications are envisioned (even though, as we have previously shown, the basic smoothing and approximation result holds in this case as well).
\end{rem}
%
%


\section*{Acknowledgement}
The author would like to thank David Xianfeng Gu for his keen interest, his stimulating questions and for 
guiding him through the implementation aspects of the combinatorial Ricci flow.


\section*{Appendix}

We include here the modicum of differential topology needed (mainly in the alternative proof of Proposition \ref{prop:BKapprox}). Our source for this material is \cite{Mun}. We presume that the reader is familiar with the basic concepts (simplicial complexes, triangulations, etc.) however, as a background text, we warmly recommend Munkres' notes \cite{Mun}.

\begin{defn}
\begin{enumerate}
\item Let $f:K \rightarrow \mathbb{R}^n$ be a $\mathcal{C}^r$ map, and let $\delta:K \rightarrow \mathbb{R}^*_+$ be a continuous function. Then $g:|K| \rightarrow \mathbb{R}^n$ is called a $\delta${\em-approximation to} $f$ iff:\\
(i) There exists a subdivision $K'$ of $K$ such that $g \in
\mathcal{C}^r(K',\mathbb{R}^n)$\,;\\
(ii) $d_{eucl}\big(f(x),g(x)\big) < \delta(x)$\,, for any $x \in |K|$\,;\\
(iii) $d_{eucl}\big(df_a(x),dg_a(x)\big) \leq \delta(a)\cdot
d_{eucl}(x,a)$\,, for any $a \in |K|$ and for all $x \in \overline{St}(a,K')$.
\item Let $K'$ be a subdivision of $K$, $U = \raisebox{0.05cm}{\mbox{$\stackrel{\circ}{U}$}}$, and let $f \in
\mathcal{C}^r(K,\mathbb{R}^n), \; g  \in \mathcal{C}^r(K',\mathbb{R}^n)$.  g is called a
$\delta${\em-approximation} of $f$ (on $U$) iff conditions (ii) and (iii) above hold for any $a \in
U$.

(Here $St(a,K)$ denotes, as it standardly does, the {\it star} of $a$ (in $K$), i.e. $St(a,K) = \bigcup_{a \in \sigma, \sigma \in K}\sigma$.
\end{enumerate}
\end{defn}

Recall that in the $PL$ context the differential (of a map) is defined as follows:

\begin{defn}
Let $\sigma$ be a simplex, and let $f:\sigma \rightarrow \mathbb{R}^n,\; f \in \mathcal{C}^r$. If $a \in \sigma$
we define $df_a:\sigma \rightarrow \mathbb{R}^n$ as follows: $df_a(x) = Df(a)\cdot (x - a)$, where $Df(a)$ denotes
the Jacobian matrix $Df(a) = (\partial f_i/\partial x^j)_{1 \leq i,j \leq n}$, computed with respect to some
orthogonal coordinate system contained in $\Pi(\sigma)$, where $\Pi(\sigma)$ is the hyperplane determined by
$\sigma$. The map $df_a:\sigma \rightarrow \mathbb{R}^n$ does not depend upon the choice of this coordinate
system.


Moreover, $df_a|_{\sigma \cap \tau}$ is well defined, for any $\sigma, \tau \in \overline{St}(a,K)$.
Therefore the map $df_a:\overline{St}(a,K) \rightarrow \mathbb{R}^n$ is well-defined and continuous, and it is
called -- analogous to the case of differentiable manifolds, the differential of $f$.
\end{defn}

%
%
%

\begin{defn}
Let $K'$ be a subdivision of $K$ and let $f \in \mathcal{C}^r(K,\mathbb{R}^n), \; g \in
\mathcal{C}^r(K',\mathbb{R}^n)$ be non-degenerate mappings (i.e. $rank(f|_{\sigma}) = rank(g|_{\sigma}) =
dim\,\sigma$,  for any $\sigma \in K$) and let $U = \raisebox{0.05cm}{\mbox{$\stackrel{\circ}{U}$}} \subset |K|$.
The mapping $g$ is called an $\alpha${\it-approximation} (of $f$ on $U$) iff:
\begin{equation}
\angle\big(df_a(x),dg_a(x)\big) \leq \alpha\,; \;{\rm for\; any}\; a \in U,\; {\rm and \; any}\; x \in
\overline{St}(a,K'),\; a \neq x.
\end{equation}
\end{defn}

As expected, a fine enough $\delta$-approximation is also an $\alpha$-approximation:

\begin{lem}[\cite{Mun}, Lemma 8.7]
Let $K$ be a (finite) simplicial complex and let $f:K \rightarrow \mathbb{R}^n$ be a non-degenerate $\mathcal{C}^r, 1 \leq r \leq \infty$ map.
Then, for any $\alpha > 0$, there exists $\delta = \delta(\alpha) > 0$ such that any non-degenerate $\mathcal{C}^r$ map $g:K' \rightarrow \mathbb{R}^n$, , which is a $\delta$-approximation of $f$ on some open set $U$, is also an $\alpha$-approximation of $f$ on $U$. (Here $K'$ denotes, as above, a subdivision of $K$.)
\end{lem}

We conclude the appendix with the following definition:

\begin{defn}\label{s:Secant}
Let $f \in \mathcal{C}^r(K)$ and let $s$ be a simplex, $s < \sigma \in K$. Then the linear map: $L_s:s \rightarrow
\mathbb{R}^n$, defined by $L_s(v) = f(v)$ where $v$ is a vertex of $s$, is called the {\em secant map induced by}
$f$.
\end{defn}



\end{document}